\documentclass[10pt]{amsart}
\usepackage{graphicx,amssymb,amsfonts,amsmath,amsthm,newlfont}
\usepackage{epsfig}
\usepackage[all]{xy}
\xyoption{poly} \xyoption{arc}
\usepackage{diagrams}
\usepackage{comment}
\usepackage{color}
 \diagramstyle[labelstyle=\scriptstyle]
\usepackage{url}

\newtheorem{thm}{Theorem}
\newtheorem{cor}[thm]{Corollary}
\newtheorem{lem}[thm]{Lemma}
\newtheorem{prop}[thm]{Proposition}

\theoremstyle{definition}

\theoremstyle{remark}
\newtheorem{rem}[thm]{Remark}
\newtheorem{feq}{FE}

\newcommand{\Z}{\mathbb Z}
\newcommand{\C}{\mathbb C}

\newcommand{\Q}{\mathbb Q}

\newcommand{\CH}{\mathcal{H}}

\newcommand{\Alt}{\mathrm{Alt}}
\newcommand{\Sym}{\mathrm{Sym}}

\newcommand{\CS}{\mathcal{S}}

\def \Spec{{\textrm{Spec}}}

\def\tr{\textcolor{red}}
\def\tb{\textcolor{black}}

\def \C{\mathbf C}
\def \CB{\mathcal B}
\def \CG{\mathcal G}

\def \beq {\begin{equation*}}
\def \eeq {\end{equation*}}
\def \beql {\begin{equation}}
\def \eeql {\end{equation}}
\def \bea {\begin{eqnarray*}}
\def \eea {\end{eqnarray*}}

\def \Z{{\mathbf Z}}

\def \Q{{\mathbf Q}}

\def \C{{\mathbf C}}
\def \PP{{\mathbf P}}

\def \CL{{\mathcal L}}

\def \CS{{\mathcal S}}
\def \cS{{\mathcal S}}
\def \C{{\mathbf C}}

\def \PGL{\text {PGL}}
\def \SymAlt{\text {SymAlt}}

\author{Herbert Gangl}

\title[The Grassmannian and Goncharov's complex in weight 4]{The Grassmannian complex and Goncharov's motivic complex in weight 4}

\begin{document}
\maketitle
\begin{abstract}For a field $F$ 
and a given integer $n>1$, Goncharov has given a complex $\Gamma_F(n)$ which he calls 
motivic and which he expects to rationally compute the weight $n$ motivic cohomology of $\Spec \,F$,  hence 
also its algebraic $K$-groups in Adams weight $n$, and he was furthermore led to---conjecturally quasi\-iso\-morphic---`thickened' complexes thereof.

These complexes involve tensor products of higher polylogarithm groups, 
the latter having been linked to the geometry of certain configurations in Goncharov's proof of Zagier's Polylogarithm Conjecture for weight~3, and 
an analogous picture has long been envisioned by Goncharov for higher weight as well \cite{GoncharovConfigurations}.

We provide a partial morphism in weight~4 by giving three out of four maps for configurations in general position. 
We also check that an associated integrability condition for the leftmost one of these maps holds.
This note was inspired very much by Goncharov's work, 
some of which is published \cite{GoncharovGeomTrilog}, some unpublished \cite{GoncharovWeight4}. In these papers he already gave partial
answers, in particular he suggested a map (corresponding to the role of $f_7(4)$ below) that is compatible with the Aomoto
polylogarithm setting. As our maps differ from his ones, we hope that our considerations are still of independent value.
\end{abstract}

\section{Introduction}
Let $F$ be a field (we mostly think of $\C$, a number field $F$ or a rational function field $F(t_1,\dots,t_r)$ over the latter).
We will consider configurations of $2n$ points in $\PP^{n-1}$ as investigated by Suslin and Goncharov. More concretely, we
will focus on the case $n=4$, and hence will  
look at configurations of eight points in $\PP^3(F)$, i.e. equivalence classes modulo the simultaneous action of the 
group $G=\PGL_4(F)$, and denote the $G$-sets of $N$ points in $\PP^3(F)$ by $C_N(4)$. 
More precisely, we consider the free abelian group on those configurations and denote it by $C_N(4)$.
There is a natural differential $C_N(4)\to C_{N-1}(4)$ defined by alternatingly leaving out one of the vectors.

Our goal is to relate the resulting complex to Goncharov's so-called motivic complex $\Gamma_F(4)$  (e.g.~\cite{GoncharovConfigurations}, \cite{GoncharovGeomTrilog}) which is defined with the help of certain single-valued variants $\CL_n(z)$ of the classical polylogarithms  $Li_n(z)= \sum_{n>0} z^k/k^n$.
The idea is to use the higher polylogarithm groups
$\CB_n(F) = \Z[F]/\langle\text{all functional equations of }\CL_n\rangle$ \  (some authors call them ``higher Bloch groups'', which unfortunately is also being used by others for a much smaller subgroup thereof),
together with their tensor products/wedge products, relating them by certain coboundary maps---ultimately inspired by a cobracket in an associated `motivic Lie algebra' originally envisaged by Beilinson, the construction of which is not established in general but is known to exist for a number field $F$ \cite{GoncharovConfigurations}---to form a complex that computes certain (graded pieces of) algebraic $K$-groups and hence the associated motivic cohomology groups.

There are some intricate combinatorial problems that make it rather hard to produce a morphism that relates the two complexes, where the most right hand map is rather straightforward.
We will be able to define two of the other three maps involved in such a way that the corresponding squares are indeed commutative.

A correct definition of the fourth map making the remaining square on the left commutative should give a ``quadruple ratio'' akin to Goncharov's triple ratio in 
the weight~3 case and would be key in a particularly satisfactory proof of said conjecture.
We are able to give an---alas only partial---solution of this problem in that we can make the diagram 
commutative for certain degenerate configurations but not yet for generic ones (work in progress with D.~Radchenko). 

Moreover, we  show in \S\ref{integrability} that $f_7(4)\circ d$ maps $C_8(4)$ to $B_3(F)\otimes F^\times$ (i.e. its contribution to $\bigwedge{}^{\hskip -2pt 2\hskip 2pt} B_2(F)$ vanishes). We also check that this combination satisfies an integrability condition and hence is expressible in terms of weight~4 hyperlogarithms.
The latter  
allows to associate to each element in $C_8(4)$ a combination of weight~4 hyperlogarithms (although we are not able to do this explicitly yet).

According to Dan \cite{Dan}, one can reduce any weight~4 hyperlogarithm to an explicit combination of the function $Li_{3,1}(x,y) = \sum_{0<k<\ell} x^k y^\ell/ k^3\ell$  (or rather its cousin $I_{3,1}(z_1,z_2)=Li_{3,1}(z_2/z_1,1/z_2)$ arising from its integral representation) and $Li_4$. Furthermore, he reduces any combination of $I_{3,1}$-terms with vanishing $\bigwedge{}^{\hskip -2pt 2\hskip 2pt} B_2(F)$--component to a sum of combinations $I_{3,1}\big(V(x,y),z)\big)$
where $V(x,y)$ denotes the five term relation for the dilogarithm. Finally, a long-standing conjecture of Goncharov held that one can write any of the latter combinations explicitly in terms of $Li_4$ alone, and this was recently solved by this author in \cite{GanglMPLweight4}, Thm 17.

Overall the above provides the existence of a map $C_8(4)$ to $B_4(F)$ but does not give an explicit form for it.


A more complete picture in weight~4, including the connection to algebraic K-theory and in particular a proof of Zagier's conjecture, has been announced recently in talks by Goncharov and Rudenko\footnote{Update: a preprint appeared today on the arXiv:1803.08585.}. 

\section{Towards a morphism of complexes}
\noindent
{\bf Notation.} For a field $F$, we will be looking at many cross ratios (of four points on the projective line over $F$), triple ratios (of six points in a projective plane), projected cross ratios and triple ratios as well as $4\times4$-determinants where the columns arise from points in affine 4-space $F^4$. 

As a shorthand, we adopt Goncharov's notation from \cite{GoncharovGeomTrilog}. 
Configurations in $C_N(4)$ are ordered sets of $N$ points $v_i\in F^4$ in general position viewed up to the diagonal action of the general linear group $GL_4(F)$. 

In particular, an expression consisting of four indices $(i_1\,i_2\,i_3\,i_4)$ is shorthand for $\Delta(v_{i_1},v_{i_2},v_{i_3},v_{i_4})= |v_{i_1}\,v_{i_2}\,v_{i_3}\,v_{i_4}|$, i.e. to the  determinant of the $4\times 4$-matrix whose columns are given by the points $v_{i_k}$. 

Our expressions will typically depend on the vectors modulo the scalar action by $F^\times$ (sometimes possibly not quite obviously so), and we then view the points also in $\PP^3(F)$.

Furthermore, an ordered set of four points $x_i$ in $\PP^1(F)$ in general position has a well-known invariant, their cross-ratio $\text{r}_2$, which we define as 
$$\text{r}_2(x_1,x_2,x_3,x_4) = \frac{(x_1-x_3)(x_2-x_4)}{(x_1-x_4)(x_2-x_3)}\,,$$
which we also abbreviate further by $(1\,2\,3\,4)$ where the $x_i$ are understood.\\ 
For an ordered set of six points $P_j$ ($j=1,\dots,6$) in $\PP^2(F)$ in general position Goncharov has invented the triple ratio, which is the antisymmetrisation under the symmetric group  $\CS_6$ of the following expression $\text{r}_3'$, again denoting $\Delta$ by $|\cdot |$ for short,
$$\text{r}_3'(P_1,P_2,P_3,P_4,P_5,P_6) = \frac{|P_1 P_2 P_4|\,\cdot\, |P_2 P_3 P_5|\,\cdot\, |P_3 P_1 P_6|}{|P_1 P_2 P_5|\,\cdot\, |P_2 P_3 P_6|\,\cdot\, |P_3 P_1 P_4|}\,,$$
which we also abbreviate by $(P_1 P_2 P_3 P_4 P_5 P_6)$, or even, provided the context is clear, simply by the index vector $(1\,2\,3\,4\,5\,6)$\,.

The only other abbreviations we are using are\\
1)  the projected cross ratio of six points in $F^4$, denoted by sequences of six indices separated by a  bar where we project four points (written to the right of the bar) from two points (written to the left of the bar) onto any generic plane in $F^4$, so with the above shorthand we have (one recognises the formula for the cross ratio after simply fixing the two vectors indicated by $i_1$ and $i_2$)
$$(i_1\,i_2\, | \,i_3\,i_4\,i_5\,i_6) = \frac{|i_1 i_2 i_3 i_5|\,|i_1 i_2 i_4 i_6|}{|i_1 i_2 i_3 i_6|\,|i_1 i_2 i_4 i_5|}\,;$$
2) the projected triple ratio term of seven points in $F^4$, denoted by sequences of seven indices separated by a  bar where we project six points (written to the right of the bar) from a seventh point (written to the left of the bar) onto any generic hyperplane in $F^4$, so with the above shorthand we have (one recognises the formula for the triple ratio term after simply fixing the vector indicated by $i_1$)
$$(i_1 \,| \,i_2\,\,i_3\,i_4\,i_5\,i_6\,i_7) = \frac{|i_1 i_2 i_3 i_5|\,|i_1 i_3 i_4 i_6|\,  |i_1 i_4 i_2 i_7|}{|i_1 i_2 i_3 i_6|\,|i_1 i_3 i_4 i_7|\,  |i_1 i_4 i_2 i_5|}\,.$$

Finally, a subscript $(\dots)_n$ for $n=2,3$ is shorthand for $\{(\dots)\}_n$ indicating a generator in $\CB_n(F)$, i.e. it is viewed modulo functional equations of $\CL_n$.

By $\Alt_n\big(f(v_1,\dots,v_n)\big)$, for a function $f$ on $n$ points in $F^4$, we understand the alternation under the symmetric group $\CS_n$. Note that we adopt the convention that we do not divide by $n!$.

\subsection{Defining maps from configurations to Goncharov's polylogarithmic motivic complex}
With the notations as in the introduction we can draw the diagram as follows, where the lower left hand group $\CG_4(F)$ is
defined as a quotient of $\Z[F] \oplus \bigwedge{}^{\hskip -2pt 2\hskip 2pt}\Z[F]$ by a certain group of relations that plays no role in the following.

\vskip 10pt
\begin{diagram}\label{morphism}
C_8(4)& \rTo^{\ d_8} & C_7(\tb{4})& \rTo^{d_7} &C_6(\tb{4})& \rTo^{d_6\ } & C_5(\tb{4}) \\
\dTo^{f_8(4)}& &\dTo^{f_7(\tb{4})}&& \dTo^{f_6(\tb{4})}&&\dTo^{f_5(\tb{4})} \\
\CG_4(F)& \rTo^{\partial} & {B_3(F)\otimes F^\times \atop \oplus\bigwedge\hskip-2pt{}^{2}B_2(F)}& \rTo^{\partial}& B_2(F)\otimes \bigwedge{}^2 F^\times &\rTo^{\partial}&\  \bigwedge{}^4 F^\times
\end{diagram}
\vskip 10pt

{\em Ignoring torsion.} Note that some of the maps involved have denominators (with only small primes involved) so we should strictly speaking tensor the groups by $\Z[1/N]$ for a certain small $N$ (in our formulas below the only denominator is $7$ and $N=7!$ should certainly suffice). For ease of notation we take this as understood.

\subsubsection{The map $f_5(4)$}
Here the vertical maps are defined as follows: the right hand one is simply given by
\begin{equation*}
f_5(4) : (01234) \mapsto \Alt_{(01234)} \Big( (0123)\wedge (0124) \wedge (0134) \wedge (0234)\Big)\,.
\end{equation*}

\subsubsection{The map $f_6(4)$}
The map $f_6(4)$ can be given as a linear combination of 3 orbits, with coefficients 2, $-1$ and 5, respectively.

More precisely, we have
\begin{equation*}
f_6(4): (1,\dots,6) \longmapsto \Alt_6\left(
   \begin{matrix}  {\phantom{+} 2\ (12|3456)_2 \otimes (1234)\wedge (1235)}\\ 
              {-1\ (12|3456)_2 \otimes (1234)\wedge (3456)}\\ 
              {+5\ (12|3456)_2 \otimes (1234)\wedge (1345)} 
     \end{matrix}
         \right).
\end{equation*}              

Note that this combination is not unique---see \S\ref{f64rewrite} below.
     
\subsubsection{The map $f_7(4)$}
Again there are ambiguities in choosing orbits under the alternation, the possibly most convenient
one being the following. We define  
$$f_{22}\big((1,2,3,4,5,6,7)\big) = \Alt_7 \left((12|3456)_2 \wedge (34|1257)_2 \right)$$
and then put

\begin{equation*}
f_7(4): (1,\dots,7) \longmapsto  \frac37 f_{22} + 2 f_{31}\,,
 \end{equation*}  
 where
\begin{equation*}
f_{31}\big((1,2,3,4,5,6,7)\big) = \Alt_7\big( (1|234567)_3\otimes(2356)\big).
  \end{equation*}              

\subsubsection{The map $f_8(4)$} The definition for this map obtained so far for certain {\em degenerate} configurations are somewhat complicated and not essential for the discussion here and will be treated elsewhere. 

\subsubsection{A compatibility condition for $f_8(4)$} 
A first condition to check at this stage is a kind of `integrability' for the symbols arising from the map $f_7(4)$ when composed with the boundary map $d$ on the configuration complex (where $d$ amounts to taking the alternating sum of configurations arising from the original one leaving out one of the vectors), i.e. that the map $\delta \cS $, where $\delta$ denotes a certain cobracket in some Lie coalgebra, applied to the image $f_7(4)\circ d$ vanishes for each generator of $C_8(4)$. This `integrability' morally guarantees that one can find a map $f_8(4)$ with the prescribed properties but it does not give an explicit candidate.

\smallskip
It turns out that it is equivalent but somewhat less cumbersome to pass to the dual picture where $d$ is replaced by the other Grassmannian differential $d'$ (which amounts to projecting from one of the eight points of a configuration in $C_8(4)$ to a sum of such in $C_7(3)$, i.e., to configurations of 7 points in $\PP^2$).

In \S\ref{integrability} below we give an---old and  rather computational---proof that $\ f_7(3) \circ d'$ vanishes under $\delta \cS $. An independent computer check of this result has also been obtained by Radchenko.


\subsection{Relating the maps $f_m(4)$}
\subsubsection{Relating $f_5(4)$ and $f_6(4)$}\label{f64rewrite}
There is an equivalent way to write the map $f_6(4)$ using the following claim.
\begin{lem}
$$\Alt_6\left((12|3456)_2\otimes (1345)\wedge (1234)^3\cdot (2345)\cdot(3456)\right) =0\,.$$
\end{lem}
\begin{proof} 
From a similar analysis as used in the proof of Theorem \ref{leftsquarethm} below
we find that alternating the term $(12|3456)_2 \otimes (1345)\wedge(3456) $ gives 
$$\Alt_6\big((12|3456)_2 \otimes (1345)\wedge(3456) \big) = 
\Alt_6\big( 3 (24|43|35|56)  - 6\,(26|36|46|56)\big)\,.$$
which we write in the shorthand below as $[3,-6,0]$. Now use from the proof of that theorem that the alternations of
the terms $3 (12|3456)_2\otimes (1345)\wedge (1234)$ and $(12|3456)_2\otimes (1345)\wedge (2345)$
can be written in the same shorthand as  $3\,[1,2,1]$ 
and $[-6,0,-3]$, respectively.
\end{proof}

\begin{cor}
We can write  $f_6(4)$ slightly differently as
\begin{equation*}
f_6(4): (1,\dots,6) \longmapsto \Alt_6\left(
   \begin{matrix}  {\phantom{+} 2\ (12|3456)_2 \otimes (1234)\wedge (1356)}\\ 
              {+1\ (12|3456)_2 \otimes (1345)\wedge (2345)}\\ 
              {+2\ (12|3456)_2 \otimes (1234)\wedge (1345)} 
     \end{matrix}
         \right).
  \end{equation*}              
\end{cor}

Considering the composite maps  \ $C_6(4) \to B_2(F)\otimes \bigwedge\hskip-2pt{}^2F^\times$, we find the following statement showing that 
the right hand square of \eqref{morphism} commutes.
\begin{thm}\label{leftsquarethm} $\partial\circ f_6(4)= - f_5(4)\circ d$.
\end{thm}

\begin{proof}
We will introduce a convenient notation by associating to a term $(abcd)$ its ``complement/dual"
inside $(123456)$, so e.g. $(1346) \leftrightsquigarrow (25)$, $(5436) \leftrightsquigarrow (12)$ etc.
Note that we are allowed to reorder the terms (e.g.~we can replace $(bacd)$ by $(abcd)$ up to invoking a possible sign 
which we can safely ignore, as it only affects 2-torsion).

\smallskip
1. {\em The first term:} $\Alt_6\big(2\ (12|3456)_2 \otimes (1234)\wedge (1356) \big)$.\\
For the first of the three terms in the alternating sum, we consider the boundary of 
$(12|3456)_2$, i.e. $ (12|3546)\wedge (12|3456)\in F^\times\wedge F^\times$, 
which can be written as 
$$\big((1235) - (1236) + (1246)-(1245)\big) \wedge \big((1234) - (1236) + (1256)-(1254)\big)\,. $$
After wedging this expression with $(1234)\wedge(1356)$ and passing to the ``complements" as 
indicated above, we get the following contributions (for convenience we also replace wedges by bars:

\begin{align*}
+46|56|56|24 \ +\ 46|34|56|24\ -\ 46|45|56|24\ -\ 46|36|56|24\phantom{\,.}\\
-45|56|56|24 \ -\ 45|34|56|24\ +\ \tr{45|45|56|24}\ +\ \tr{45|36|56|24}\phantom{\,.}\\
+35|56|56|24 \ +\ 35|34|56|24\ -\ 35|45|56|24\ -\ 35|36|56|24\phantom{\,.}\\
-36|56|56|24 \ -\ 36|34|56|24\ +\ \tr{36|45|56|24}\ +\ \tr{36|36|56|24}\,.
\end{align*}

Now the terms in the first column of this tableau  can be ignored as we have a factor $56\wedge 56$ in each.
The first and last of the red terms vanish for similar reasons while the remaining two red terms 
cancel due to the antisymmetry property of the wedge ($45|36|\dots +36|45|\dots = 0$). So we are left with eight terms
\begin{align*}
 \ +\ 46|34|56|24\ -\ 46|45|56|24\ -\ 46|36|56|24\\
\ -\ 45|34|56|24\ \ \phantom{+\ \tr{45|45|56|24}\ +\ \tr{45|36|56|24}}\\
 \ +\ 35|34|56|24\ -\ 35|45|56|24\ -\ 35|36|56|24\\
 \ -\ 36|34|56|24\ \ \phantom{+\ \tr{36|45|56|24}\ +\ \tr{36|36|56|24}}
\end{align*}
which we bring into a kind of normalised form 
by permuting the four factors, possibly invoking a sign which we mark below in red

\begin{align*}
 \ +\ 24|34|46|65\ -\ 24|45|46|65\ \tr{+}\ 24|46|36|65\\
\ -\ 24|34|45|56\  \phantom{ +\ \tr{45|45|56|24}\ +\ \tr{45|36|56|24}}\\
 \ +\ 24|43|35|56\ -\ 24|45|53|56\ \tr{+}\ 24|35|36|56\\
 \ -\ 24|34|36|56\ \phantom{ +\ \tr{36|45|56|24}\ +\ \tr{36|36|56|24}}\,.
\end{align*}

Under the alternation $\Alt_6$ some of the terms cancel while others are identified.
More precisely, leftover terms 1, 3, 4 and 6 are identified (having the same sign), 
as are leftover terms 5 and 8. Term 2 is fixed under the odd permutation $(13)$, while term 7
is fixed under the odd permutation $(24)$; hence the $\Alt_6$--orbits of the latter two vanish.

Overall, the boundary of the first term 
(i.e. of $\Alt_6\big((12|3456)_2 \otimes (1234)\wedge (1235)\big)$) 
becomes 
$$4 \big(24|34|46|65\big) + 2 \big( 24|43|35|56) \,.$$

\medskip
2. {\em The second term:}  $\Alt_6 \big( {+1\ (12|3456)_2 \otimes (1345)\wedge (2345)}\big)$.\\
For the second term, we find similarly
\begin{align*}
+46|56|56|26 \ +\ 46|34|56|26\ -\ 46|45|56|26\ -\ 46|36|56|26\\
-45|56|56|26 \ -\ 45|34|56|26\ +\ \tr{45|45|56|26}\ +\ \tr{45|36|56|26}\\
+35|56|56|26 \ +\ 35|34|56|26\ -\ 35|45|56|26\ -\ 35|36|56|26\\
-36|56|56|26 \ -\ 36|34|56|26\ +\ \tr{36|45|56|26}\ +\ \tr{36|36|56|26}
\end{align*}
giving the ``normalised form" 
\begin{align*}
 \ \tr{-}\ 34|46|65|26\ \ \tr{+} \ 45|46|65|26\ -\ 26|36|46|56\\
 \ \tr{+}\ 34|45|56|62\ \phantom{\ +\ \tr{45|45|56|26}\ +\ \tr{45|36|56|26}}\\
 \ \tr{-}\ 43|35|56|62\ -\ 35|45|56|26\ -\ 35|36|56|26\\
 \ \tr{+}\ 34|36|56|26.\ \phantom{\ +\ \tr{36|45|56|26}\ +\ \tr{36|36|56|26}}
\end{align*}
In the latter expression, leftover terms 4 and 5 are identified, 
while terms 1 and 6 cancel, and terms 2 and 7 are each fixed under an 
odd permutation, hence their alternation vanishes. So we obtain for the boundary of the second term, 
i.e. of $\Alt_6\big((12|3456)_2 \otimes (1234)\wedge (3456)\big)$, the contribution
$$ \big(34|36|56|26\big) +  2 \big(34|45|56|62\big)  -\big(26|36|46|56\big)\,.$$

\medskip
3. {\em The third term:}  $\Alt_6 \big(  {+2\ (12|3456)_2 \otimes (1234)\wedge (1345)}  \big)$.\\
Finally, for the third term we obtain
\begin{align*}
+46|56|26|16 \ +\ 46|34|26|16\ -\ 46|45|26|16\ -\ 46|36|26|16\\
-45|56|26|16 \ -\ 45|34|26|16\ +\ \tr{45|45|26|16}\ +\ \tr{45|36|26|16}\\
+35|56|26|16 \ +\ 35|34|26|16\ -\ 35|45|26|16\ -\ 35|36|26|16\\
-36|56|26|16 \ -\ 36|34|26|16\ +\ \tr{36|45|26|16}\ +\ \tr{36|36|26|16}
\end{align*}
and in ``normalised form" 
\begin{align*}
\tr{-}16|26|46|56\, \ \tr{-}\ 34|46|26|16\ \tr{+}\ 45|46|26|16\ -\ 16|26|36|46\\
-45|56|26|16 \ -\ 45|34|26|16\ \phantom{+\ \tr{45|45|26|16}\ +\ \tr{45|36|26|16}}\\
+35|56|26|16 \, +\, 35|34|26|16\, -\, 35|45|26|16\ -\ 35|36|26|16\\
-16|26|56|36\, \ \tr{+}\ \,34|36|26|16.\ \phantom{+\ \tr{36|45|26|16}\ +\ \tr{36|36|26|16}}
\end{align*}
This time the terms in the first column are non-zero, and we have three different
types of terms, all of which arise with compatible sign. Leftover terms 6, 8 and 9
are fixed under an odd permutation each, hence do not contribute.
Terms 1, 4 and 11 combine to multiplicity 3, while the remaining six terms
combine to multiplicity 6, so we obtain for the boundary of the third term the contribution
$$- 6 \big(34|46|26|16\big) -3 \big(16|26|46|56\big) \,.$$

4. {\em Combining the three contributions.} \\
Now it remains to compare the terms under the alternation. Putting 
$$[a,b,c]= \Alt_6\big( a \big(24|34|46|65\big) + b \big( 24|43|35|56\big) + c \big(26|36|46|56\big)\big)\,,$$
we find $[4,2,0]$, $[-1,-2,-1]$ and $[-6,0,-3]$ for the terms in 1., 2.~and 3., respectively.

\medskip
Noting that $2[4,2,0] +2[-1,-2,1] +[-6,0,-3] = [0,0,-1]$, we see that the diagram commutes.
\end{proof}

\begin{rem}
Certain results concerning maps $f_k(4)$ which are very related to the above have been obtained by Goncharov in \cite{GoncharovGeomTrilog}
and \cite{GoncharovWeight4}, 
but so far we were not able to reconcile our results with the terms and  coefficients given there.
\end{rem}

\subsubsection{Relating $f_6(4)$ and $f_7(4)$ $($the middle square of  \eqref{morphism}$)$}
We have, as maps  \ $C_7(4) \to B_3(F)\otimes F^\times\phantom{\bigg|}$, the following statement showing that 
the centre square of \eqref{morphism} commutes.

\begin{thm} $ f_6(4)\circ d=\partial\circ f_7(4)$.
\end{thm}
\begin{proof}
{\em LHS.} 
We first compute the left hand side for an arbitrary configuration of seven vectors.
We rewrite below the three contributions under the Alt--sign using the permutations in cycle form $(2\,5\,4)$, $(2\,3)$ and $(3\,1\,4\,2\,5)$, respectively,
to obtain

\bea
f_6(4)\circ d \big((123467)\big) &=& \phantom{+} 
\,2 \,\Alt_7 (12|3456)_2 \otimes  (1234) \,\wedge\, (1345)\\
&& + \,2\,\Alt_7 (12|3456)_2 \otimes  (1234) \,\wedge\, (1356)\\
&&+ \ \ \Alt_7 (12|3456)_2 \otimes (1345) \,\wedge\, (2345)\,.
\eea
We rewrite below the three contributions under the Alt--sign using the permutations $(2\,5\,4)$, $(2\,3)$ and $(3\,1\,4\,2\,5)$, respectively,
to obtain
\bea
f_6(4)\circ d \big((123467)\big) &=&
\phantom{+}\, 2 \,\Alt_7 (15|3246)_2 \otimes  (1532) \,\wedge\, (1324)\\
&& + \,2\,\Alt_7 (13|2456)_2 \otimes  (1234) \,\wedge\, (1256)\\
&&+ \ \ \Alt_7 (45|1236)_2 \otimes (4123) \,\wedge\, (5123)\,.
\eea

Note that the terms in the respective wedge products come in two types:
they either overlap in two entries (the second sum) or in three entries (first and last sum). Hence it is useful to collect 
the ones of the same type; in fact we also apply for later convenience the permutation $(1\,3)(6\,7)$ to the first sum.
Moreover, we note the stabilisers that will fix the expressions or turn them  into their negatives:
for the last sum, which we can write as
$$ \ \Alt_7 (45|1237)_2 \otimes (1235) \,\wedge\, (1234)$$
(note that we swapped the terms in the wedge product but also applied the odd permutation $(6\,7)$) hence keep the sign).
This expression is symmetric in the indices $4,5$ and antisymmetric in the three indices $1,2 , 3$ as well as in the indices
$6, 7$; so we can think of regrouping terms in the sum in this manner and combine first and third sum to
$$\Sym_{(4\,5)}\Alt_{\langle (1\,2),(1\,3), (6\,7)\rangle}\Big(\big((2 (35|1247) + (45|1237)\big)\otimes 1235\wedge 1234\Big)\ +\ \text{symm.}$$
where ``$+$ symm.'' denotes the sum over the different $\Alt_7$--translates of that expression. 
Similarly, the second sum can be written as
$$\Alt_{(3\,5)(4\,6)}\Alt_{\langle(1\,2),(3\,4),(5\,6)\rangle} \big((13|2456)\otimes(1256)\wedge (1234)\big) \ +\ \text{symm.}$$
where this time the antisymmetrization is over a group of order 4, generated by the two involutions (given in the indices of Alt)  whose cycle forms
are $(3\,5)(4\,6)$ and $(1\,2)(3\,4)(5\,6)$, respectively.
It will suffice to match these regrouped terms in order to prove commutativity of the square.

In summary, the left hand side can be written 
\bea
\hskip -15pt&\hskip -10pt =&\hskip -5pt \Sym_{(4\,5)}\Alt_{\langle (1\,2),(1\,3), (6\,7)\rangle}\Big(\big((2 (35|1247) + (45|1237)\big)\otimes 1235\wedge 1234\Big)
\ +\ \text{symm.}\\
&&+\Alt_{(3\,5)(4\,6)}\Alt_{\langle(1\,2),(3\,4),(5\,6)\rangle} \Big((13|2456)\otimes(1256)\wedge (1234)\Big) \ +\ \text{symm.}
\eea

\medskip
{\em RHS.}
We will now analyse the two expressions arising from taking the boundary of the two
terms in $f_7(4)$. 

The first term is $\partial f_{31}\big((1234567)\big)$.
By factoring we get $(1|234567)= \frac{|1235|\cdot |1346|\cdot |1427|}{|1236|\cdot |1347|\cdot |1425|}$ \ where the factors are $4\times4$--determinants  and so we find
\begin{lem}
\bea
\partial   \Alt_7\Big((1|234567)_3\otimes (2356)\Big)&\hskip -10pt =&\hskip -10pt 
\Alt_7 (1|234567)_2\otimes\Big( (1235)\wedge (2356)  -(1236)\wedge (2356)\\
&&\phantom{\Alt_7((1|234567)} +(1346)\wedge (2356) - (1347)\wedge (2356)\\
&&\phantom{\Alt_7((1|234567)} +(1427)\wedge (2356) - (1425)\wedge (2356)\Big)\,.
\eea
\end{lem}
\noindent (Proof is straightforward, using that $(1|234567)= \frac{|1235|\cdot |1346|\cdot|1427|}{|1236|\cdot |1347|\cdot|1425|}$.)

\medskip
Grouping again by types of ``overlaps'' of indices of the two tensor factors on the right, we find that the first two terms $(1235)\wedge (2356)$ and $(1236)\wedge (2356)$ have three overlapping indices ($\{2,3,5\}$ and $\{2,3,6\}$, respectively) and are combined into the first line below,
terms \#3 and \#6 have overlap~2 and are combined into the second line below, and finally terms \#4 and \#5 have overlap~1 and are combined into the third line below. In summary, we can write the first term on the RHS ($=\partial\circ   \Alt_7\big((1|234567)_3\otimes (2356)\big)$)  as

\bea
&&\hskip -18pt 2\,\Sym_{(4\,5)}\Alt_{\langle (1\,2),(1\,3), (6\,7)\rangle}\Big( (5|236147)_2 + (5|236417)_2\Big)\otimes (1235)\wedge (1234) +\text{symm.}\\
&& \hskip -25pt -2 \,\Alt_{(3\,5)(4\,6)}\Alt_{(1\,2)(3\,4)(5\,6)} \Big( (3|614527)_2  + (5|236147)_2 \Big)\otimes (1256)\wedge(1234)+\text{symm.}\\
&&\hskip -25pt -4 \,\Alt_7 \Big((1|234567)_2\otimes (1347)\wedge(2356)\Big) \,.
\eea

It turns out that the last line (i.e. the contribution with overlap~1) vanishes.

\begin{prop} We have
$$\Alt_7 \Big((1|234567)_2\otimes (1347)\wedge(2356)\Big)  = 0\,.$$
\end{prop}
\begin{proof}
We expand the expression in a similar way to what we did for the above terms into
$$\Alt_{(1\,2)(4\,5)(6\,7)}\Alt_{\langle (1\,4),(4\,7),(2\,5),(5\,6)\rangle}  \Big((1|234567)_2\otimes (1347)\wedge(2356)\Big)   \ +\ \text{symm.}$$
and then are reduced to show that the individual sums already vanish. For this we use \eqref{feqtripleratio} below.
\end{proof}

Similarly, the second term $\partial\circ  \frac37  \Alt_7\Big((12|3456)_2\otimes (34|1257)_2\Big)$  of the RHS  is by definition
\beq
\frac37 \, \Alt_7\Big((12|3456)_2\otimes (34|1257)\wedge (34|1527) \ -\ (34|1257)_2\otimes (12|3456)\wedge(12|3546)\Big)
\eeq
and we can rewrite it as follows.

\begin{prop} The expression \  $- \partial\circ  \frac37  \Alt_7\Big((12|3456)_2\otimes (34|1257)_2\Big)$ \ equals
\beq
 2 \, \Sym_{(4\,5)}\Alt_{\langle (1\,2),(1\,3), (6\,7)\rangle}
 \Big(\big((45|1237)_2 + 2(46|1237)_2\big)\otimes (1254)\wedge(1234)\Big)  +\ \mathrm{symm.}
\eeq
\end{prop}
\begin{proof}
The differential of $(x)_2\wedge(y)_2 \in  \bigwedge{}^{\hskip -2pt 2\hskip 2pt} B_2(F)$ is defined as $(x)_2\otimes (y)\wedge(1-y) \ - \ (y)_2\otimes (x)\wedge(1-x)
\in B_2(F)\otimes \bigwedge{}^{\hskip -2pt 2\hskip 2pt} F^\times$, so up to the involution $(1\,3)(2\,4)(6\,7)$ we obtain
twice the same terms, so using that $1- \text{r}_2(x_1,x_2,x_3,x_4) = \text{r}_2(x_1,x_3,x_2,x_4)$ we obtain
$$-\partial\circ  \frac37  \Alt_7\Big((12|3456)_2\otimes (34|1257)_2\Big) = 2\,(34|1257)_2\otimes (12|3456)\wedge
(12|3546)$$
and the latter can be written as
$$ 2\,\Alt_7  
(34|1257)_2 \otimes \begin{pmatrix} 
\phantom{+}(1235)\wedge(1234)\\ 
+(1235)\wedge (1256)\\
-(1235)\wedge (1236)\\
-(1235)\wedge (1254)\\
+(1246)\wedge(1234)\\ 
+(1246)\wedge (1256)\\
-(1246)\wedge (1236)\\
-(1246)\wedge (1254)\\
-(1236)\wedge(1234)\\ 
-(1236)\wedge (1256)\\
-(1245)\wedge(1234)\\ 
-(1245)\wedge (1256) 
\end{pmatrix} = 
  2\,\Alt_7  \begin{pmatrix} 
\phantom{+}(34|1257)_2 \\ 
{+(56|1237)_2} \\
{+(36|1257)_2} \\
{+(54|1237)_2}\\
{+(34|1267)_2}\\
{+(64|1257)_2}\\
{+(54|1267)_2}\\
{+(63|1257)_2}\\
{+(34|1267)_2}\\
{+(46|1257)_2}\\
{+(53|1247)_2}\\
{+(64|1237)_2}
\end{pmatrix}  \otimes (1235)\wedge(1234)
$$
where we apply the following permutations, respectively: 
ii) (3\,5)(4\,6);  iii) (4\,6); iv) (3\,5); v) (3\,4)(5\,6); vi) (3\,6); vii) (6\,3\,5); viii) (6\,4\,3); ix) (5\,6); x) (6\,3); xi) (3\,4); xii) (6\,5\,3).
 
\smallskip
Note that terms \#1 and \#11 give the same orbit as do terms \#2 and \#12, terms \#3 and \#8, terms \#5 and \#9
as well as terms \#6 and \#10, so we we are left with only seven different terms as follows
$$
  2\,\Alt_7  \begin{pmatrix} 
\phantom{+}2 (34|1257)_2 \\ 
{+2(56|1237)_2} \\
{+2(36|1257)_2} \\
{+(54|1237)_2}\\
{+2(34|1267)_2}\\
{+2(64|1257)_2}\\
{+(54|1267)_2}\\
\end{pmatrix}  \otimes (1235)\wedge(1234)
$$
Now we need to use a couple of functional equations to write each of these summands (or rather their $\Sym_{(4\,5)}\Alt_{\langle (1\,2),(2\,3),(6\,7)\rangle}$--orbits) in terms of a generating orbit set.
It turns out that we need three such orbits, we choose as representatives  $(45|1237)_2$, $(46|1237)_2$ and $(46|1257)_2$. As a shorthand, we will then express each of the six summands as a linear combination of the given three generators and identify each with the corresponding coefficient vector. E.g. $(34|1257)_2 \leftrightsquigarrow [\frac13,\frac13,-1]$ 
denotes 
$$\SymAlt(34|1257)_2  = \SymAlt\big( \frac13 (45|1237)_2 + \frac13 (46|1237)_2 - (46|1257)_2\big),$$
where SymAlt is shorthand for $\Sym_{(4\,5)}\Alt_{\langle (1\,2),(2\,3),(6\,7)\rangle} $.
Then we find the following correspondences
\bea
\phantom{+}2 (34|1257)_2 & \leftrightsquigarrow & [\frac23,\frac23,-2]\\ 
{+2(56|1237)_2} & \leftrightsquigarrow & [0,2,0]\\
{+2(36|1257)_2} & \leftrightsquigarrow & [0,\frac23,0]\\
{+(54|1237)_2}& \leftrightsquigarrow & [0,\frac43,0]\\
{+2(34|1267)_2}& \leftrightsquigarrow & [1,0,0]\\
{+2(64|1257)_2}& \leftrightsquigarrow & [\frac23,0,0]\\
{+(54|1267)_2}& \leftrightsquigarrow &  [0,0,2]
\eea
which arise from the following functional equations (possibly only valid under SymAlt):
\bea
3(45|1267)_2 &=& 2(45|1237)_2\,,\\
3(34|1267)_2 &=& 2 (46|1237)_2\,,\\
3(36|1257)_2 &=& (46|1237)_2\,,\\
3(34|1257)_2&=& (45|1237)_2 + (46|1237)_2 - 3(46|1257)_2\,,\\
3(35|1267)_2 &=& 2 (46|1237)_2\,.
\eea
Adding up the right hand vectors give $- \frac73[1,2,0]$ which then immediately translates into the claim of the proposition.
\end{proof}

\smallskip
In order to finish the proof of the theorem, we need to compare the three different types of contributions according to the number of overlapping indices (either one, two or three) in the respective rightmost two wedge factors.


\begin{enumerate}
\item 
For \#overlaps$=1$ there is no contribution from the RHS, so the contribution from the LHS has to vanish. Indeed we have 

\begin{lem}
\beq \label{feqtripleratio}
\Alt_{(1\,2)(4\,5)(6\,7)}\Alt_{\langle (1\,4),(4\,7),(2\,5)(5\,6)\rangle}(1|234567)_2 = 0\,.
\eeq
\end{lem}
(Proof see below, FE2.)
\item 
For \#overlaps$=2$ there are contributions from the LHS and from the first term of the RHS. Demanding that their 
difference vanishes amounts to the following statement.

\begin{lem}
\beq
\Alt_{(3\,5)(4\,6)}\Alt_{\langle (1\,2),(3\,4),(5\,6)\rangle}\big((13|2456)_2 + (3|614527)_2 +(5|236147)_2\big) = 0\,.
\eeq
\end{lem}
(Proof see below, FE3.)

\item 
For \#overlaps$=3$ there are contributions from the LHS and both terms on the RHS. Demanding that their 
difference vanishes amounts to the following statement.

\begin{lem}
\bea
&&\Sym_{(4\,5)}\Alt_{\langle (1\,2),(2\,3),(6\,7)\rangle}\big(  (35|1247)_2- (46|1237)_2\\
&&\phantom{\Sym_{(4\,5)}\Alt_{\langle (1\,2),(2\,3),(6\,7)\rangle}} - (5|236147)_2-(5|236417)_2\big) = 0\,.
\eea
\end{lem}
(Proof see below, FE4.)
\end{enumerate}

Invoking all these three relations we obtain that each of the contributions on the LHS equals the corresponding 
contribution of the RHS, which proves the theorem.
\end{proof}

\begin{rem}
There is actually some 
ambiguity involved in the choice of $f_7(4)$, and it seems possible that there is a more
suitable candidate for it.
In particular, in Goncharov's analysis of the morphism \eqref{morphism} using Aomoto polylogarithms
there are other orbits involved in his definition of $f_7(4)$. On the other hand, our choice seems to be a decent one, too,
as the small coefficients of our combination are rather reassuring.
\end{rem}

\section{Vanishing of the $\bigwedge^{\hskip -2pt 2\hskip 2pt} B_2$--component and an integrability condition}\label{integrability}\label{integrability}

\subsection{The $\bigwedge^{\hskip -2pt 2\hskip 2pt} B_2$--component}
\medskip
When dealing with $C_8(4)$, and in particular when dealing with determinants of $4\times 4$-matrices, duality simply replaces any 4-element subset of $\{1,2,\dots,8\}\,$ of indices by the complementary set of indices, while preserving the multiplication and division of determinants. We still ignore torsion (in particular 2-torsion) in the following, so we can neglect signs of these determinants and hence write sequences of four indices representing such determinants in the natural ascending order. Furthermore, in each expression for $f_{31}\circ d'$ there is one index that will occur for all the determinants involved, indicating that this indexes the point from where we have projected; 
we drop this common index which then gives a map on $C_7(3)$.

\subsubsection{The part from $f_{31}$.} We now consider the term corresponding to the dual of the second contribution $f_{31}$. In particular we  will deal with $3\times3$--determinants in this dual situation which makes the task slightly less cumbersome. This amounts to analysing the symbol attached to
the alternating sum $\Alt_7$  of the following expressions where each number $1,\dots,7$ stands for an associated point in $\PP^2$
$$ (123456)_3\otimes 147\,,
$$
and where $(\dots)_3$ is a shorthand for the triple ratio of the six points, projected to the higher Bloch group $B_3(F)$.\\
For a triple ratio $a$, the numerator of the factorisation of $(1-a)$ factors into a $3\times 3$-determinant and a $6\times 6$-determinant. This indicates that we have to consider two types of factors here.

{\em Type 1:} In the decomposition of $1-x$, where $x$ is one of Goncharov's triple ratios, there are `new factors' arising (certain $6\times 6$-determinants). For each such `new factor' the contribution to $\ \cS\circ f_7(3) \circ d'$ is zero. In order to see this, note that anti-symmetrising with respect to the group of order eight arising from the generators $(1\,4)$, $(2\,5)$ and $(3\,6)$ fixes the new factor as well as the right hand tensor factor $|147|$, and the corresponding eight terms add up to zero in the same way as for Goncharov-Zagier's 840-term relation for the trilogarithm.

{\em Type 2:} All the other factors are $4\times 4$--determinants. \\
Let us write the tensor symbol (to $\{a\}_3\otimes b$ we antisymmetrise $(1-a)\otimes a\otimes a \otimes b$ with respect to the first two slots, i.e.~we associate $(1-a)\otimes a\otimes a \otimes b\ -\ a\otimes (1-a) \otimes a \otimes b$) attached to the typical expression that is leftover when we remove the Type 1-factors, and expand it into suitable building blocks which constitute elementary tensor products. Hence each tensor factor of the latter is simply a single such $4\times 4$--determinant.

\medskip
After dropping type 1 factors we are left with considering the $\Alt_7$--alternation of 
$$\Big(\frac {|123|\,\phantom{|236|\,|314|}}{|125|\,|236|\,|314|} \wedge \frac {|124|\,{|235|\,|316|}}{|125|\,|236|\,|314|}\Big) \otimes  \frac {|124|\,{|235|\,|316|}}{|125|\,|236|\,|314|} \otimes \,|147|\,.$$

After expanding in this way we are dealing with 36$(=(3+3)\times (3+3))$ terms containing the leftover numerator $|123|$ of the leftmost tensor factor,
together with 54$(=3\times 3\times (3+3))$ terms containing instead one of the denominator factors of that same tensor factor. (Note that, due to the antisymmetry in the first two slots, we can ignore terms involving only denominator factors from the first two slots.)

Among the $\Alt_7$--orbits of the resulting 90 expressions there are only 34 which are actually non-zero. We will consider those in more detail.
Two orbits occur with multiplicity 5, four others occur with multiplicity 2, the remaining 16 ones only occur once.

For completeness' sake we reproduce representatives for all the leftover orbits, together with their multiplicity (we drop the determinant bars and replace tensor or wedge signs by a comma for ease of notation):

$$
\begin{matrix}
\phantom{+}5 (123, 124, 135, 246) [*],\ &
+5 (123, 124, 135, 356)[\text{X}],\  \\
+2 (123, 124, 125, 146)[\text{2A}],\ &
-2 (123, 124, 125, 156)[-\text{2B}],\ \\
+2 (123, 124, 135, 146)[\text{2E}],\ &
+2 (123, 124, 135, 156) [*],\ \\
+ (123, 124, 135, 126)[\text{A}],\ &
+ (123, 124, 135, 136)[\text{B}],\ \\
+ (123, 124, 145, 126)[\text{A}],\ &
+ (123, 124, 145, 146)[\text{B}],\ \\
+ (123, 124, 145, 236) [*],\ &
+ (123, 124, 145, 456)[\text{X}],\ \\
+ (123, 145, 124, 126) [\text{A}],\ &
+ (123, 145, 124, 146)[-\text{A}],\ \\
+ (123, 145, 124, 236) [\text{F}],\ &
+ (123, 145, 124, 456)[-\text{F}],\ \\
- (123, 145, 246, 137)[\text{G}],\ &
- (123, 145, 246, 157)[-\text{G}],\ \\
- (123, 145, 246, 237)[\text{H}],\ &
- (123, 145, 246, 457)[-\text{H}],\ \\
- (123, 145, 246, 267)[\text{I}],\  &
- (123, 145, 246, 467)[-\text{I}]\,.
\end{matrix}
$$

Applying the map $\delta$ essentially  boils down to replacing a four-fold tensor product $a\otimes b\otimes c\otimes d$ with the eightfold combination arising from alternating the first two slots, alternating the last two slots, and furthermore alternating under the swap (slot 1$\leftrightarrow$ slot 3, slot 2$\leftrightarrow$ slot 4). 

Three orbits among the above 34 ones, marked with a $[*]$, vanish under $\delta$: apply the cycle   $(1\,2)(3\,4)(5\,6)$ to both the multiplicity~5 orbits $(123, 124, 135, 246)$ and the multiplicity~1 orbit $(123, 145, 145, 236)$, and apply the cycle $(2\,5)(4\,6)$ to the multiplicity~2 orbit $(123,124,135,156)$.

Furthermore, under $\delta$ some of the remaining 26$(=34-5-2-1$) orbits agree, possibly up to sign only, and we are left with 8 orbit types only, denoted by roman letter (A, B, E, F, G, H, I, X) as follows, where we indicate the contributions from multiplicities by a superscript. The three orbits marked F, G and H in the three last lines above cancel pairwise (use the permutation $(2\,4)(3\,5)$ in each case).

\smallskip
Type A: $(+^2)$, $(+)$,  $(+)$,  $(+)$,  $(-)$,  overall multiplicity 4($=2+1+1+1-1$);

Type B: $(-^2)$, $(+)$,  $(+)$, overall multiplicity 0;

Type E: $(+^2)$, overall multiplicity 2; 

Type F, G, H, I: $(+)$, $(-)$, overall 0;

Type X: $(+^5)$, $(+)$, overall multiplicity 6.


\smallskip
Finally, since Type A and Type E consist of expressions whose four factors contain a common index, they vanish once we compose with the boundary map $d'$ as the latter 
provides each factor with a second common index and  the transposition swapping these two indices is an odd permutation fixing the expression.

\smallskip
{\em Upshot:} The only type that will contribute to $\delta\circ \cS\circ f_{31}\circ d'$ is type $X$, and it occurs with coefficient 6.

\medskip
\subsubsection{The part from $f_{22}$.} 
We  consider the term corresponding to the dual of the second contribution $f_{22}$. This amounts to analysing the $\Alt_7$--alternation of
$$
\bigg(\frac{|123|\,|145|}{|125|\,|143|} \,\wedge \frac{|124|\,|135|}{|125|\,|134|}\bigg) \ \wedge \bigg( \frac{|215|\,|267|}{|217|\,|265|} \,\wedge \frac{|216|\,|257|}{|217|\,|256|}\bigg)\,.
$$
which gives us $12^2=144$ terms (again, we do not need to consider contributions if the two leftmost factors---or the two rightmost factors---both arise from the denominator).

Under $\Alt_7$ there are 89 non-zero orbits leftover, which are grouped into orbits of multiplicities  9 (2 such), 7 (5 such), 5 (2 such), 2 (7 such) and 1 (12 such).

 $$
 \begin{matrix} 
 9( 123,124,456,145), [-\text{X}] &
  9(123,124,356,135), [-\text{X}] \\
  7(123,124,156,157), [\text{D}] &
  -7(123,124,156,125),[\text{B}]  \\
  7(123,124,145,146), [\text{B}] &
  7(123,124,135,136), [\text{B}] \\
  7(123,124,125,156), [\text{B}] &
  -5(123,124,145,456), [-\text{X}] \\
  -5(123,124,135,356), [-\text{X}] &
  -2(123,124,156,257) [*], \\
  2(123,124,156,145) [*], &
  2(123,124,156,135) [*], \\
  -2(123,124,145,246) [*], &
  2(123,124,145,156) [*], \\
  -2(123,124,135,236) [*], &
  2(123,124,135,156) [*], \\
  123,124,456,157, [\text{I}] &
  123,124,356,157,[\text{I}]  \\
  123,124,345,146,[\text{C}]  &
  123,124,345,136, [\text{C}] \\
   -(123,124,156,457), [-\text{I}] &
   -(123,124,156,357), [-\text{I}] \\
   -(123,124,145,346), [-\text{C}] &
   -(123,124,145,126), [-\text{A}] \\
   -(123,124,135,346), [-\text{C}] &
   -(123,124,135,126), [-\text{A}] \\
  123,124,125,146, [\text{A}] &
  123,124,125,136, [\text{A}] \,.
 \end{matrix}   
$$

After applying $\delta$ precisely the seven multiplicity 2 orbits (marked by $[*]$) vanish.

Furthermore, under $\delta$ some of the remaining 75$(=89-2\cdot 7)$ orbits agree, possibly up to sign only, and we are left with 6 orbit types only (four of which agree with orbit types for $f_{31}$).

\smallskip
Type A: $(+)$, $(+)$,  $(-)$,  $(-)$, overall multiplicity 0;

Type B: $(+^7)$, $(+^7)$,  $(+^7)$  $(+^7)$, overall multiplicity 28;

Type C, I: $(+)$, $(+)$, $(-)$, $(-)$, overall multiplicity 0; 

Type D, $(+^7)$, overall 7;

Type X: $(-^9)$, $(-^9)$, $(-^5)$, $(-^5)$, overall multiplicity $-28$.


\smallskip

Finally, since Type B and Type D consist of expressions whose four factors contain a common index, they vanish under $d'$ as above.

\smallskip
{\em Upshot:} The only type that will contribute is type X, and it occurs with coefficient $-28$, so combining with the result above it is now clear how to cancel the type X contributions. "Dualising the indices" as indicated above we get the following theorem. Note that the linear combination given is 12 times the map $f_{7}(4)$ above.

\bigskip
\begin{thm} The following linear combination vanishes under $\delta\circ \cS$:
$$\Alt_8\Bigg(28 \Big((1|234567)_3\otimes (1258)\Big) \ + \  6 \Big( (81|2345)_2 \wedge (82|1567)_2\Big)\Bigg)\,.$$
Furthermore, $f_7(4)\circ d$ maps any configuration in $C_8(4)$ to $B_3(F)\otimes F^\times$, i.e.~its $\bigwedge^{\hskip -1pt 2\hskip -2pt} B_2$---contribution vanishes.
\end{thm}

\subsection{Integrability} We now consider the above orbits after applying $d'$, which essentially amounts to adding a common index to each of the four factors. Note that all the resulting orbits, with the possible exception of type X orbits, either cancel each other or they vanish under the combined antisymmetrisation under $\Alt_8$ and under swapping the first two or the last two tensor factors. Using the criterion for integrability as adapted from Chen, as e.g.~given in \cite{DuhrGanglRhodes} (3.17), one can check that also the type X orbit is indeed integrable to a weight~4 multiple polylogarithm. \\
For the $f_{31}$--part, we need to apply the map sending $\{a\}_3\otimes b\in B_3(F)\otimes F^\times$ to $\big((1-a)\otimes a \, -\, a\otimes (1-a)\big)\otimes \big( d \log a\wedge d\log b\big)$ and check that the image of the type X orbit vanishes.\\
For the $f_{22}$--part, we have to apply $\delta$ to $\{a\}_2\wedge \{b\}_2$ and further map it (up to overall sign) to the tensor product of $a\wedge b$ and the sum of four symmetric tensors, with rational functions as coefficients, given by $$\frac1{ab} \big((1-a)\odot (1-b)\big) + \frac1{a(1-b)} \big((1-a)\odot b\big) + \frac1{(1-a)b} \big(a\odot (1-b)\big)+\frac1{(1-a)(1-b)} (a\odot b)\,.$$ Here we have used the notation  $x\odot y = x\otimes y \,+\,y\otimes x$ for the symmetric tensor.
The package \cite{PolylogTools} can be used to check that the corresponding antisymmetrised expressions vanish: for the former one only needs to check the integrability with respect to the last two tensor factors as the other possibilities vanish, anyway, while for the latter  only the two middle factors have to be tested. The cancellations are rather non-trivial and rely on (differences of) projected Pl\"ucker relations.

\bigskip \noindent
{\bf Remark:} The results from the above paragraph imply that we can attach to each configuration in $C_8(4)$ a weight~4 hyperlogarithm.
As already indicated in the introduction, Dan \cite{Dan} showed that any hyperlogarithmic expression in weight~4 can be explicitly reduced\footnote{A couple of misprints in his formula were corrected in collaboration with Duhr, for details see  \cite{Charlton}, Thm.~5.2.5 and Rem.~5.2.6.} to one in $I_{31}$ and $Li_4$. He also showed that  the exactness of the complex $0 \rightarrow B_4(F)_\Q \rightarrow \CH_4(F)_\Q \rightarrow^{\delta_4} \bigwedge{}^{\hskip -2pt 2\hskip 2pt} B_2(F)_\Q \rightarrow0$, where $\CH_4(F)_\Q$ denotes the vector space of formal hyperlogarithms in weight~4 and $\delta_4$ an associated coboundary map, would follow provided one can show a conjecture of  Goncharov stating that  $I_{31}(V(x,y), z)$, where $V(x,y)$ denotes the five term relation for the dilogarithm, can be expressed in terms of $Li_4$ only. A solution to this conjecture was given in \cite{GanglMPLweight4}, Thm 17.
Therefore we conclude that we can attach to each configuration in $C_8(4)$ a linear combination of $Li_4$ terms, i.e. 
a map $f_8(4)$ (an explicit version of which is still elusive) completing the left hand square. This provides a further stepping stone towards Zagier's Conjecture in weight~4, a proof of which has been announced by Goncharov and Rudenko.
Moreover, with Radchenko we have obtained partial results pertaining to certain {\em degenerate} configurations which allow us to define $f_8(4)$ explicitly in those situations.

\bigskip
{\bf Acknowledgements.} In order to derive and check the results given in this paper we used Goncharov's symbol for iterated integrals as it was implemented in Mathematica by Duhr \cite{PolylogTools} for our joint paper \cite{DuhrGanglRhodes}. The author is very grateful for the hospitality in particular of the Max-Planck-Institute and also the Hausdorff-Institute for Mathematics in Bonn where the maps and proofs were obtained during our stay in 2013 and subsequently refined in 2015 and 2018.

\section{Appendix: Functional equations.}

In this section we collect and prove a couple of functional equations for the dilogarithm needed in the proof that the centre square
of \eqref{morphism} commutes.

\begin{lem}\label{feqlemma}
There are the following symmetries among the terms $(\dots)_2$:
\begin{enumerate}
\item \label{crsym}
$(12|3456)_2 $ is symmetric with respect to the permutation $(1\,2)$ and 
antisymmetric with respect to permutations of 3,4,5,6, so e.g.
$(12|3456)_2  = - (12|4356)_2 = (12|4536)_2$ \ etc. 
\smallskip
\item  \label{trsym}
$(1|234567)_2 = (1|342675)_2 = (1|423756)_2  = - (1|324657)_2 = -  (1|432765)_2 = - (1|243576)_2\,.$
\smallskip
\item  The (projected) triple ratio $(1|234567)$ is a product of two (doubly) projected cross ratios in the following three ways:
$$(1|234567) = \frac{(12|3457)}{(13|2467)} = \frac{(13|4265)}{(14|3275)} =\frac{(14|2376)}{(12|4356)}\,.$$
\item \label{ftr4} We have a five term relation involving two triple ratios and two projected cross ratios:
$$(1|234567)_2 - (1|237564)_2 = -\Big(\frac{(12|3574)}{(13|2674)} \Big)_2 + (13|2764)_2 - (12|3754)_2\,.$$

\noindent $[$Proof: write $x=(13|2764)$ and $y=(12|3754) $, then the five terms are expressed as follows:
$(1|234567)=\frac{(12|3457)}{(13|2467)} =\frac{1-x^{-1}}{1-y^{-1}}$,\\
$(1|237564)=\frac{(12|3754)}{(13|2764)} =\frac{y}{x}$ and $\frac{(12|3574)}{(13|2674)}= \frac{1-y}{1-x}$, so
the above reduces to the five term relation in the form
$$ \Big(\frac{1-x^{-1}}{1-y^{-1}}\Big)_2 -  \Big(\frac{y}{x}\Big)_2 =  -\Big(\frac{1-y}{1-x}\Big)_2 + (x)_2-(y)_2\,.]$$

\end{enumerate}
Furthermore, we note that the terms of the form $(a\,b\,c\,d)$ occurring in a tensor or wedge factor are a shorthand for a $4\times 4$--determinant and hence are, up to 2-torsion, invariant under permutation.
\end{lem}

\begin{feq} Variants of the five term relation.
\beql\label{ftr1}
2 \,\Alt_{\langle (6\,7),(1\,2\,3)\rangle} (56|1237)_2 = 3 \Alt_{\langle (6\,7),(1\,2\,3)\rangle}(52|3716)_2\,. 
\eeql

\beql \label{ftr2}
\Alt_{(6\,7)} (5|236147)_2 = - (53|2746)_2 +(52|3716)_2 \,.
\eeql
\end{feq}


\begin{feq}
\beql \label{feqtripleratio}
\Alt_{(1\,2)(4\,5)(6\,7)}\Alt_{\langle (1\,4),(4\,7),(2\,5)(5\,6)\rangle}(1|234567)_2 = 0\,.
\eeql
\end{feq}
\noindent 
(We skip the proof which is similar to but easier than the ones for {\it FE}~3 and {\it FE}~4 below.)

\begin{feq}
\beq
\Alt_{(3\,5)(4\,6)}\Alt_{\langle (1\,2),(3\,4),(5\,6)\rangle}\big((13|2456)_2 + (3|614527)_2 +(5|236147)_2\big) = 0\,.
\eeq
\end{feq}
\begin{proof}
The second term $(3|614527)_2$ can be written via Lemma \ref{feqlemma} \eqref{trsym}, as  $(3|146275)_2 = \Big(\dfrac{31|4625}{34|1675}\Big)_2$
and invoking Lemma \ref{feqlemma} \eqref{ftr4} we find
$$\Alt_{(5\,6)} (3|146275)_2 = -\Big(\dfrac{31|4256}{34|1756}\Big)_2 + (34|1576)_2 - (31|4526)_2\,.$$
Similarly, the third term $(5|236147)_2$ can be rewritten under the alternation sign as $-(3|254167)_2$ (use the
permutation $(3\,5)(4\,6)$) and hence via Lemma \ref{feqlemma} \eqref{trsym} as $-(3|425716)_2$ and we get an analogous five term relation 
$$-\Alt_{(5\,6)} (3|425716)_2 = \Big(\dfrac{34|2765}{32|4165}\Big)_2 - (32|4615)_2 +(34|2675)_2\,.$$
Now the first terms on the right in the above two five term equations turn out to be negatives of each other, hence they cancel in the sum. (One can also check that Alt-orbit of that first term on the right vanishes as it is invariant under the odd permutation $(1\,2)$: write it out as a product of determinants to get $\frac{|3126|\cdot |3475|}{|3125|\cdot|3476|}$.)

\smallskip
Therefore the original sum can be replaced by the sum 
\bea
\Alt_{(3\,5)(4\,6)} &&\hskip -25pt\Alt_{\langle (1\,2),(3\,4),(5\,6)\rangle}\\
\hskip -30pt \bigg((13|2456)_2 &+& \frac12 \big((34|1576)_2 - (31|4526)_2\big) 
+\frac12 \big(-(32|4615)_2 +(34|2675)_2\big) \bigg).
\eea
But $\Alt_{(1\,2)} (34|1576)_2 = -\Alt_{(1\,2)} (34|2576)_2 = \Alt_{(1\,2)} (34|2675)_2$ (for the second equality we use
the symmetries in Lemma \ref{feqlemma} \eqref{crsym} and similarly $\Alt_{(1\,2)} (31|4526)_2 =  \Alt_{(1\,2)}  (32|4615)_2$,
so we can combine the five summands to three, all with coefficient $\pm1$, in fact to
$$\big((13|2456)_2 + (34|1576)_2 - (31|4526)_2\big) $$
but the first and last of these cancel in view of  Lemma \ref{feqlemma} \eqref{crsym} while the middle one is invariant under the odd permutation $(3\,4)$ and hence its Alt-orbit vanishes. In summary, the original sum indeed vanishes as claimed.
\end{proof}

\begin{feq}
\bea
&&\Sym_{(4\,5)}\Alt_{\langle (1\,2),(2\,3),(6\,7)\rangle}\big(  (35|1247)_2- (46|1237)_2\\
&&\phantom{\Sym_{(4\,5)}\Alt_{\langle (1\,2),(2\,3),(6\,7)\rangle}} - (5|236147)_2-(5|236417)_2\big) = 0\,.
\eea
\end{feq}
\begin{proof}
Adding \eqref{ftr2} to its variant 
where 1 and 4 are swapped we get
$$- (5|236147)_2-(5|236417)_2 = -\frac12 (53|2746)_2 +\frac12 (52|3716)_2 -\frac12 (53|2716)_2+\frac12 (52|3746)_2$$
and so, alternating with respect to $(2\,3)$ we find 
$$\Alt_{(2\,3)} \big(- (5|236147)_2-(5|236417)_2\big) =\Alt_{(2\,3)}\big(  - (53|2746)_2 - (53|2716)_2\big)\,.$$
Furthermore, antisymmetrising with respect to $(4\,5)$ and then invoking \eqref{ftr1}  we find
$$\Alt_{(4\,5)} \big(-(46|1237)_2 \big)=\Alt_{(4\,5)} \big((56|1237)_2\big) = \frac32 \Alt_{(4\,5)} (53|2716)_2\,,$$
so we can now write the combination in question
\bea
&&\Sym_{(4\,5)}\Alt_{\langle (1\,2),(2\,3),(6\,7)\rangle}\big(  (35|1247)_2- (46|1237)_2\\
&&\phantom{\Sym_{(4\,5)}\Alt_{\langle (1\,2),(2\,3),(6\,7)\rangle}} - (5|236147)_2-(5|236417)_2\big) \\
&=&\Sym_{(4\,5)}\Alt_{\langle (1\,2),(2\,3),(6\,7)\rangle}\big(  (35|1247)_2+\frac32  (53|2716)_2\\
&&\phantom{\Sym_{(4\,5)}\Alt_{\langle (1\,2),(2\,3),(6\,7)\rangle}} - (53|2746)_2-(53|2716)_2\big) \,.
\eea
We combine the second and fourth term to $\frac12  (53|2716)_2$ and use that the operator $\Alt_{\langle (1\,2),(2\,3),(6\,7)\rangle}$ antisymmetrises over 
$(1\,2)$ and also over $(6\,7)$ so under the alternation sign we can replace $(35|1247)_2$ by
$\frac12 (35|1247)_2 - \frac12 (35|1246)_2$  and $ - (53|2746)_2$ by $-\frac12 (53|2746)_2 +\frac12 (53|1746)_2$.
Hence we obtain that the expression under the alternation sign reduces to a standard five term relation 
(with fixed projection points 3 and 5) as follows, with the obvious new ad hoc notation $(53| 27164)$,
$$ 0 =` \partial \big( (53| 27164)\big)\text{'} = (53|7164)_2 - (53|2164)_2+(53|2764)_2-(53|2714)_2+(53|2716)_2\,,$$
thereby proving the statement.
\end{proof}

\bibliographystyle{amsplain_initials_eprint}
\bibliography{config_gon3}
 
 \end{document}